\documentclass[11pt,reqno,a4paper]{amsart}
\usepackage{a4,amsfonts,amssymb,amsmath}
\usepackage{mathrsfs}
\usepackage{dsfont}
\usepackage{color}
\usepackage[]{graphicx}

\newcommand{\rem}[1]{}

\newtheorem{theorem}{Theorem}[section]

\theoremstyle{definition}
\newtheorem{definition}[theorem]{Definition}

\newtheorem{example}[theorem]{Example}

\renewcommand\leq{\leqslant}
\renewcommand\geq{\geqslant}

\title[Infinite-dimensional linear port-Hamiltonian systems]{Infinite-dimensional linear port-Hamiltonian systems on a one-dimensional spatial domain: An Introduction}

\author{Birgit Jacob \and Hans Zwart}

\begin{document}
	\maketitle

\begin{abstract}
We provide an introduction to infinite-dimensional port-Hamil\-tonian systems.
As this research field is quite rich, we restrict ourselves to the class of infinite-dimensional linear port-Hamiltonian systems on a one-dimensional spatial domainand we will focus on topics such as Dirac structures, well-posedness, stability and stabilizability, Riesz-bases and dissipativity. We combine the  abstract operator theoretic approach with  the more physical approach based on Hamiltonians. This enables us to derive easy verifiable conditions for well-posedness and stability. 
\end{abstract}

\maketitle

\section{Introduction}
Systems described by partial differential equations (PDEs) can be investigated either by operator theoretic or PDE methods. The PDE methods are specialized to specific classes of PDEs, and therefore lead to  refined results. The operator theoretic methods formulate the main concepts and investigate their interconnections. The advantage of the operator theoretic approach is that it allows for a general abstract framework.
In this survey, we combine the  abstract operator theoretic approach with  a more physical approach based on Hamiltonians in order to   derive easy verifiable conditions for well-posedness and stability of port-Hamiltonian systems.

Many physical systems can be formulated using a Hamiltonian framework. This class of systems
contains ordinary as well as partial differential equations. Each system in this class has a Hamiltonian, generally given by the energy function. In the study
of Hamiltonian systems it is usually assumed that the system does not interact with its
environment. However, for the purpose of control and for the interconnection of two or
more Hamiltonian systems it is essential to take this interaction with the environment into
account. This led to the class of port-Hamiltonian systems, see
\cite{vS06,ScMa02}. The Hamiltonian/energy has been used to control a port-Hamiltonian system,
see e.g.\ \cite{BCEJLLM09,CeSB07,HDLM10,OvSME02}. For port-Hamiltonian systems described by ordinary differential equations this approach is very successful, see the references mentioned above. Port-Hamiltonian systems described by partial differential equation is a subject of current research, see e.g.\ \cite{EbMS07,JeSc09,KZSB10,MM05}.


\section{Example of a port-Hamiltonian system}
\label{examples}

Various control systems can be modeled by partial differential equations such as vibrating strings, flexible structures, the propagation of sound waves, and networks of strings or flexible structures.  Here we consider the simple example of a transmission line with boundary control and observation. 
\begin{example}\label{ex1}
Transmission lines with  boundary controls  are used for electric power transmission.
\begin{figure}[b]
\centering
\includegraphics[width=9cm]{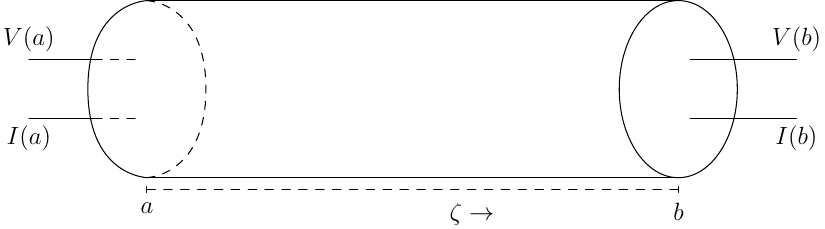}
\caption{Schematic representation of the transmission line}\label{TLsc}
\end{figure}
The problem can be approximated by the 1D system of Figure \ref{TLsc} representing propagation of electric charges and magnetic fluxes.  

The {\em lossless transmission line} on the spatial interval $[a,b]$ is described by the PDE:
  \begin{align}
  \label{eq:7.1.5}
    \frac{\partial Q }{\partial t}(\zeta,t) &= -\frac{\partial}{\partial \zeta} \frac{\phi(\zeta,t)}{L(\zeta)}\\
  \nonumber
    \frac{\partial \phi}{\partial t}(\zeta,t) &= -\frac{\partial}{\partial \zeta} \frac{Q(\zeta,t)}{C(\zeta)}.
  \end{align}
  Here $Q(\zeta,t)$ is the charge at position $\zeta\in [a,b]$ and time $t>0$,
  and $\phi(\zeta,t)$ is the (magnetic) flux at position $\zeta$ and time $t$. $C$ is the
  (distributed) capacity and $L$ is the (distributed) inductance.
  The voltage and current are given by $V=Q/C$ and $I=\phi/L$,
  respectively. The energy of this system is given by
  \begin{equation*}
    E(t) = \frac{1}{2} \int_a^b \frac{\phi(\zeta,t)^2}{L(\zeta)} + \frac{Q(\zeta,t)^2}{C(\zeta)} d \zeta.
  \end{equation*}
 The control, boundary conditions  and observation associated to this problem are 
\begin{align*}
  V\left(a,t\right)=u(t) 
 && V\left(b,t\right)=RI\left(b,t\right)&& I\left(a,t\right)=y(t) 
\end{align*}
  corresponding to a controlled voltage at point $a$, a resistive charge at point $b$ and an observed current at point $a$. Here $u(t)$ is the control, $y(t)$ the observation and $R\ge 0$ is the resistor. 
  For the change of energy we obtain
  \begin{align}
    \nonumber
    \frac{dE}{dt} (t) &= \int_a^b \frac{\phi(\zeta,t)}{L(\zeta)}\frac{\partial \phi}{\partial t}(\zeta,t) + \frac{Q(\zeta,t)}{C(\zeta)}\frac{\partial Q}{\partial t}(\zeta,t) d \zeta\\
    \nonumber
    &=\int_a^b - \frac{\phi(\zeta,t)}{L(\zeta)}\frac{\partial}{\partial \zeta} \frac{Q(\zeta,t)}{C(\zeta)} -  \frac{Q(\zeta,t)}{C(\zeta)}\frac{\partial}{\partial \zeta} \frac{\phi(\zeta,t)}{L(\zeta)} d \zeta \\
  \nonumber
  &= - \int_a^b \frac{\partial}{\partial \zeta} \left(  \frac{\phi(\zeta,t)}{L(\zeta)}\frac{Q(\zeta,t)}{C(\zeta)} \right) d \zeta\\
  &= \frac{\phi(a,t)}{L(a)}\frac{Q(a,t)}{C(a)} - \frac{\phi(b,t)}{L(b)}\frac{Q(b,t)}{C(b)}\nonumber
 \\
  &= V(a,t)I(a,t)-V(b,t)I(b,t) \nonumber \\
  &=u(t)y(t) - RI(b,t)^2 ,\label{eq:1.3}
  \end{align}
  where we used the boundary conditions. 
  Equation (\ref{eq:1.3}) shows that the change of energy can only occur via the boundary.
  Since voltage times current equals power and the change of energy is
  also power, this equation represents a power balance.
\end{example}

\section{Class of port-Hamiltonian systems}
\label{SEC:2.4}

Many physical systems can be modelled by the following equation
\begin{align}\nonumber
  \frac{\partial x}{\partial t}(\zeta,t) &= P_1 \frac{\partial }{\partial \zeta} \left({\mathcal H}(\zeta) x(\zeta,t) \right) + P_0 \left( {\mathcal H}(\zeta) x(\zeta,t) \right), \quad \zeta\in (a,b), t>0,\\
x(\zeta,0) &= x_0(\zeta), \qquad \zeta\in (a,b), \nonumber\\
 u(t) &=\widetilde W_{B,1} \begin{pmatrix} {\mathcal
        H}(b) x(b,t) \\ {\mathcal
        H}(a) x(a,t) \end{pmatrix},\qquad t>0,\label{eq:1.2.8}\\
 0 &=\widetilde W_{B,2} \begin{pmatrix} {\mathcal
        H}(b) x(b,t) \\ {\mathcal
        H}(a) x(a,t) \end{pmatrix},\qquad t>0,\nonumber\\
        y(t) &=\widetilde W_{C} \begin{pmatrix} {\mathcal
        H}(b) x(b,t) \\ {\mathcal
        H}(a) x(a,t) \end{pmatrix},\qquad t>0.\nonumber
\end{align}
Here $P_1\in \mathbb R^{n\times n}$ is invertible and self-adjoint, i.e., $P_1^\top=P_1$, and $P_0\in \mathbb R^{n\times n}$  is  {\em skew-adjoint}, i.e., $P_0^\top=-P_0$.
We assume that ${\mathcal H}\in L^\infty((a,b);\mathbb R^{n\times n})$, for every $\zeta \in [a,b]$,
${\mathcal H}(\zeta)$ a is self-adjoint matrix and there exist constants $c,C>0$, such that $cI \leq
{\mathcal H}(\zeta) \leq C I$ for almost every $\zeta\in[a,b]$. Finally,  $\widetilde W_{B,1}$ is a $m\times 2n$-matrix, $\widetilde W_{B,2}$ is a $(n-m)\times 2n$-matrix, 
 and $\widetilde W_C$ a $m\times 2n$-matrix. 
 Here
$u(t)\in \mathbb R^m$ denotes the input  and  $y(t)\in \mathbb R^m$ the output  at time $t$. 
We call \eqref{eq:1.2.8} a {\em port-Hamiltonian system}. We remark that the operator $P_1 \frac{\partial }{\partial \zeta}  + P_0$ is formally skew-adjoint on $L^2((a,b);\mathbb R^n)$.

The energy or Hamiltonian can be expressed by using $x$ and ${\mathcal H}$. That is
\begin{equation}
  \label{eq:1.2.9}
  E(x(\cdot,t)) = \frac{1}{2} \int_a^b x(\zeta,t)^\top {\mathcal H}(\zeta) x(\zeta,t) d \zeta.
\end{equation}
In Example \ref{ex1}, the change of energy (power) of the system was only possible via the boundary of its spatial domain. In general, for the Hamiltonian given
  by \eqref{eq:1.2.9} 
  the following balance equation holds for all
  (classical) solutions of \eqref{eq:1.2.8} 
  \begin{equation}
    \label{eq:1.2.10}
    \frac{dE}{dt}(x(\cdot,t)) = \frac{1}{2} \left[ \left({\mathcal H}
  x\right)^\top\! (\zeta,t) P_1  \left({\mathcal H} x\right)(\zeta,t) \right]_a^b 
  .
  \end{equation}
This balance  equation will prove to be very important and will be useful in many problems, such as the existence of solutions and stability.
\begin{example}[Lossless transmission line] 
\label{E:2.4.3}

\mbox{}\newline
If we introduce in Example \ref{ex1} the variables $x_1=Q$ and $x_2=\phi$,  Equation \eqref{eq:7.1.5}
can be written as
\begin{equation*}
  \frac{\partial }{\partial t} \begin{bmatrix} x_1(\zeta,t) \\ x_2(\zeta,t) \end{bmatrix} =  -\frac{\partial }{\partial \zeta}\left(\begin{bmatrix} 0 & 1 \\ 1 & 0 \end{bmatrix} \begin{bmatrix} \frac{1}{C(\zeta)} & 0 \\ 0 & \frac{1}{L(\zeta)} \end{bmatrix} \begin{bmatrix} x_1(\zeta,t) \\ x_2(\zeta,t) \end{bmatrix} \right) 
\end{equation*}
which is of the form of Equation \eqref{eq:1.2.8} with
\begin{equation*}
\label{eq:p1}
P_1=\begin{bmatrix}0 & -1 \\ -1 & 0\end{bmatrix}, \quad P_0=0, \quad {\mathcal H}(\zeta)=\begin{bmatrix} \frac{1}{C(\zeta)} & 0 \\ 0 & \frac{1}{L(\zeta)}\end{bmatrix}.
\end{equation*}
The boundary condition, control and observation  can be rewritten as
\begin{align*}
\label{ex:1boundary}
  \begin{bmatrix} 0&0& 1&0\\
  1 &-R&0&0\end{bmatrix}  \begin{bmatrix} \frac{x_1}{C}\left(b,t\right) \\ \frac{x_2}{L}\left(b,t\right)\\
  \frac{x_1}{C}\left(a,t\right)\\ \frac{x_2}{L}\left(a,t\right)\end{bmatrix} &=\begin{bmatrix} u(t) \\0\end{bmatrix},\quad \begin{bmatrix} 0&0& 0&1\end{bmatrix}  \begin{bmatrix} \frac{x_1}{C}\left(b,t\right) \\ \frac{x_2}{L}\left(b,t\right)\\
  \frac{x_1}{C}\left(a,t\right)\\ \frac{x_2}{L}\left(a,t\right)\end{bmatrix} =y(t),
\end{align*}
that is, $\widetilde W_{B,1}= \begin{bmatrix} 0&0& 1&0\end{bmatrix} $, $\widetilde W_{B,2}= \begin{bmatrix} 
  1 &-R&0&0\end{bmatrix}$, $\widetilde W_C=\begin{bmatrix} 0&0& 0&1\end{bmatrix} $,  $n=2$, and $m=1$.
Finally, the Hamiltonian is written as
\[
E(x(\cdot,t)) = \frac{1}{2} \int_a^b \frac{x_1(\zeta,t)^2}{C(\zeta)} + \frac{x_2(\zeta,t)^2}{L(\zeta)} d \zeta. 
\]
\end{example}

\section{Dirac structures and port-Hamiltonian systems}
\label{sec:1.3}

In the previous section we have seen a class of partial differential equations satisfying a power balance (\ref{eq:1.2.10}). In this section, we show that there is an underlying structure capturing this conversation of energy. This structure is known as a Dirac structure, which will be defined next.
\begin{definition}
  \label{D:1.3.4}%
  Let ${\mathcal E}$ and ${\mathcal F}$ be two (linear) vector spaces which are dual to each other with duality product 
    \begin{equation}
  \label{eq:1.3.10}
    \langle f,e \rangle_{{\mathcal F},{\mathcal E}}, \quad e\in {\mathcal E}, f \in {\mathcal F}.
  \end{equation}
  The \emph{bond space} ${\mathcal B}$ is defined as ${\mathcal F} \times {\mathcal E}$. On ${\mathcal B}$ we define the following symmetric pairing
  \begin{equation}
    \label{eq:1.3.11}
    \left\langle \left(\begin{array}{c} f_1\\ e_1 \end{array}\right), \left(\begin{array}{c} f_2\\ e_2 \end{array}\right) \right\rangle_+ = \langle f_1, e_2 \rangle_{{\mathcal F},{\mathcal E}} +\langle f_2, e_1 \rangle_{{\mathcal F},{\mathcal E}} .
  \end{equation}
  Let ${\mathcal V}$ be a linear subspace of ${\mathcal B}$, then the orthogonal subspace with respect to the symmetric pairing (\ref{eq:1.3.11}) is defined as
  \begin{equation}
    \label{eq:1.3.12}%
    {\mathcal V}^{\perp} = \{ b \in {\mathcal B} \mid \langle b,v \rangle_+ = 0 \mbox{ for all } v \in {\mathcal V} \}.
  \end{equation}
  A \emph{Dirac structure} ${\mathcal D}$ is a linear subspace of the bond space  satisfying
  \begin{equation}
    \label{eq:1.3.13}
      {\mathcal D}^{\perp} = {\mathcal D}.
  \end{equation}
\end{definition}

The variables $e$ and $f$ are called the \emph{effort} and \emph{flow},
respectively, and their spaces ${\mathcal E}$ and ${\mathcal F}$ are
called the {\em effort}\index{effort space} and \emph{flow space}. The
bilinear product $\langle f, e
\rangle_{{\mathcal F},{\mathcal E}}$ is called the \emph{power} or \emph{power
  product}.
  
A simple but important fact is that the power of any element $b=(f,e)\in {\mathcal D}$ in a Dirac structure ${\mathcal D}$ is zero, which follows from  the following calculation
\begin{equation}
  \label{eq:1.3.14}%
  2  \langle f, e \rangle_{{\mathcal F},{\mathcal E}} = \langle f, e \rangle_{{\mathcal F},{\mathcal E}} + \langle f, e \rangle_{{\mathcal F},{\mathcal E}}  = \langle b, b\rangle_+ = 0,
\end{equation}
because (first) $b \in {\mathcal D}$ and (second) $b \in {\mathcal D}^{\perp}$.

A Dirac structure can been seen as the largest subspace for which this holds i.e., if ${\mathcal V}$ is a subspace of ${\mathcal B}$ satisfying (\ref{eq:1.3.14}), then ${\mathcal V}$ is a Dirac structure if there does not exists a subspace ${\mathcal W}$ such that ${\mathcal V} \subset {\mathcal W}$, ${\mathcal V} \neq {\mathcal W}$, and the power of every element in ${\mathcal W}$ is zero.

Comparing the conservation of power (\ref{eq:1.3.14}) with the balance equation of our port-Hamiltonian system (\ref{eq:1.2.10}), we notice that in our system there are two power flows. Namely the internal power and the power over the boundary. If we want that the total power remains zero, then we have to consider these two powers together. 
\begin{definition}
  Given the conditions on $P_0$, $P_1$ we define the following Dirac structure associated to our port-Hamiltonian system (\ref{eq:1.2.8}). We choose 
$${\mathcal E}={\mathcal F}=L^2((a,b);{\mathbb R}^n) \times {\mathbb R}^n$$
 with duality product
  \begin{equation}
  \label{eq:1HZ}
    \left\langle \left( \begin{matrix} f \\ f_{\partial} \end{matrix} \right), \left( \begin{matrix} e\\ e_{\partial} \end{matrix} \right)\right\rangle = \int_a^b f(\zeta)^T e(\zeta) d\zeta - f_{\partial}^T e_{\partial}.
  \end{equation}
  \begin{align}
  \nonumber
    {\mathcal D} =&\ \left\{  \left( \begin{array}{c} f \\ f_{\partial}\\ e\\ e_{\partial}  \end{array} \right) \in {\mathcal F} \times {\mathcal E} \mid f \in L^2((a,b);{\mathbb R}^n), e  \in H^1((a,b);{\mathbb R}^n) \right.\\
      \label{eq:2HZ}
    &\qquad  \left. f = P_1 \frac{de}{d\zeta} + P_0 e , \left( \begin{matrix}  f_{\partial}  \\  e_{\partial}\end{matrix} \right) =
   \frac{1}{\sqrt{2}} \left[ \begin{array}{cc} P_1 & -P_1 \\ I &  I \end{array} \right]  \left( \begin{matrix}  e(b) \\  e(a) \end{matrix} \right)
    \right\}
  \end{align}
\end{definition}

Thus this Dirac structure has internal and external variables, 
and it depicted in Figure \ref{fig:1.2}.
\begin{figure}[ht]
  \centering
   \includegraphics[scale=1.0]{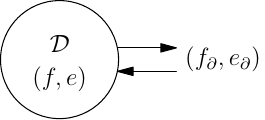}
  \caption{Dirac structure}
  \label{fig:1.2}
\end{figure}

\begin{theorem}
  Under the conditions on $P_1$ and $P_0$ the linear subspace as defined in (\ref{eq:2HZ}) is a Dirac structure under the pairing (\ref{eq:1HZ}).
\end{theorem}
\begin{proof}
  We shall not provide the full proof here, for that we refer to \cite{GoZM05}, but we will show that the power is zero, ie., ${\mathcal D} \subseteq {\mathcal D}^{\perp}$.
  
  Let $b=(f_1,f_{1,\partial}, e_1, e_{1,\partial})$ be an element of ${\mathcal D}$. To show that it is an element of ${\mathcal D}^{\perp}$, we have to show that $\langle b, v\rangle_+ =0$ for all $v = (f_2,f_{2,\partial}, e_2, e_{2,\partial}) \in {\mathcal D}$. Using (\ref{eq:1.3.11}) and (\ref{eq:1HZ}) we get
  \[
    \langle b, v \rangle_+ = \int_a^b f_1(\zeta)^T e_2(\zeta) d\zeta - f_{1,\partial}^T e_{2,\partial} + \int_a^b f_2(\zeta)^T e_1(\zeta) d\zeta - f_{2,\partial}^T e_{1\partial}.
\]
Applying (\ref{eq:2HZ}) this becomes
  \begin{align*}
     \langle b, v \rangle_+ =&\  \int_a^b \left[ P_1 \frac{de_1}{d\zeta} (\zeta)+ P_0 e_1(\zeta)\right] ^T e_2(\zeta) d\zeta -  f_{1,\partial}^T e_{2,\partial} \\
     &+\  \int_a^b \left[ P_1 \frac{de_2}{d\zeta} (\zeta)+ P_0 e_2(\zeta)\right] ^T e_1(\zeta) d\zeta -  f_{2,\partial}^T e_{1,\partial}\\
     =&\ \int_a^b  \frac{de_1}{d\zeta} (\zeta)^T P_1 e_2(\zeta)  + \frac{de_2}{d\zeta} (\zeta)^T P_1 e_1(\zeta) d\zeta -   f_{1,\partial}^T e_{2,\partial} - f_{2,\partial}^T e_{1,\partial},
  \end{align*}
  where we have used the symmetry of $P_1$ and the anti-symmetry of $P_0$. The integral term equals
  \[
    \left[e_1(\zeta)^T P_1 e_2(\zeta) \right]_{a}^b.
  \]
  Summarising we have
  \begin{align*}
    \langle b, v \rangle_+ =&\  \left[e_1(\zeta)^T P_1 e_2(\zeta) \right]_{a}^b - \frac{1}{2} \left[ P_1(e_1(b)-e_1(a)) \right]^T \left[e_2(b) + e_2(a) \right]  \\
    &\  -\frac{1}{2}\left[ P_1(e_2(b)-e_2(a)) \right]^T \left[e_1(b) + e_1(a) \right]\\
    =&\ e_1(b)^T P_1 e_2(b) -  e_1(a)^T P_1 e_2(a)  \\
    &\ -e_1(b)^T P_1 e_2(b) + e_1(a)^T P_1 e_2(a) =0,
  \end{align*}
  which proves our assertion.\hfill$\Box$
\end{proof}

The link between the Dirac structure (\ref{eq:2HZ}) and the port-Hamiltonian system is given by the following substitution:
\begin{equation}
\label{eq:3HZ}
  f = \dot{x}\, \mbox{ and }\, e = {\mathcal H} x.
\end{equation}
With this choice, we get
\begin{align*}
\nonumber
  \frac{d}{dt} \frac{1}{2} \int_a^b x(\zeta, t)^T {\mathcal H} x(\zeta,t) d\zeta =&\ \frac{1}{2} \int_a^b \dot{x}(\zeta,t)^T {\mathcal H} x(\zeta,t) + x(\zeta,t)^T {\mathcal H} \dot{x}(\zeta,t) d\zeta\\
  \label{eq:4HZ}
   =&\  \frac{1}{2}\int_a^b f^T e + e^T f d\zeta = f_{\partial}^Te_{\partial},
\end{align*}
where we used the Dirac structure (\ref{eq:2HZ}) with its duality/power product (\ref{eq:1HZ}). The power given by $f_{\partial}$ and $e_{\partial}$, can be expressed in $e$ via (\ref{eq:2HZ}), and via (\ref{eq:3HZ}) into ${\mathcal H}x$.

We want to remark that there is no time in the Dirac structure, the time enters via the choice of $e$ and $f$, see (\ref{eq:3HZ}). So the same Dirac structure could support a discrete time system.  

In our figure of a Dirac structure, Figure \ref{fig:1.2}, we have already shown the variables, $f_{\partial}$ and $e_{\partial}$ as pointing in and out of the Dirac structure. This is because Dirac structure can be connected. 
\begin{figure}[ht]
  \centering
   \includegraphics[scale=1.0]{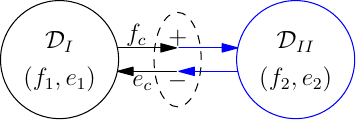}
  \caption{Coupling of two Dirac structures}
  \label{fig:1.3}
\end{figure}

Given two Dirac structures, we define the coupling of them  by (see also Figure \ref{fig:1.3})
\begin{equation}
  \label{eq:1.3.25}
  {\mathcal D} = \left\{ \left(\begin{array}{c} f_1 \\ f_2 \\ e_1 \\ e_2 \end{array} \right) \mid \exists  f_c, e_c \mbox{ s.t. }
   \left(\begin{array}{c} f_1 \\ f_c \\ e_1 \\ e_c \end{array} \right) \in {\mathcal D}_I \mbox{ and }
  \left(\begin{array}{c} f_2 \\ f_c \\ e_2 \\ -e_c \end{array} \right) \in {\mathcal D}_{II} \right\}.
\end{equation}
 This structure has zero power. To see this, we assume that for both systems $\langle f, e\rangle + \langle e, f \rangle$ denotes the power. We take as power for ${\mathcal D}$
\[
  \langle f_1, e_1 \rangle +  \langle e_1, f_1 \rangle +  \langle f_2, e_2 \rangle  + \langle e_2, f_2 \rangle. 
\]
We calculate for this power product
\begin{align*}
  \langle f_1, e_1 \rangle +   \langle e_1, f_1 \rangle +&  \langle f_2, e_2 \rangle  + \langle e_2, f_2 \rangle \\
  =&~ 
  \langle f_1, e_1 \rangle +  \langle e_1, f_1 \rangle + \langle f_c, e_c \rangle +  \langle e_c, f_c \rangle \\
  &~
  -  \langle f_c, e_c \rangle +  \langle e_c, f_c \rangle  + \langle f_2, e_2 \rangle -  \langle e_2, f_2 \rangle \\
  =&~ \left\langle \left( \begin{matrix} f_1 \\ f_c \end{matrix} \right), \left( \begin{matrix} e_1 \\ e_c \end{matrix} \right) \right\rangle + \left\langle \left( \begin{matrix} e_1 \\ e_c \end{matrix} \right), \left( \begin{matrix} f_1 \\ f_c \end{matrix} \right) \right\rangle + \\
  &~ \left\langle \left( \begin{matrix} f_2 \\ f_c \end{matrix} \right), \left( \begin{matrix}  e_2 \\ - e_c \end{matrix} \right) \right\rangle + \left\langle \left( \begin{matrix} e_2 \\-e_c \end{matrix} \right), \left( \begin{matrix} f_2\\ f_c \end{matrix} \right) \right\rangle.
\end{align*}
The last expressions are zero since $\left(\begin{smallmatrix} f_1 \\
    f_c \\ e_1 \\ e_c \end{smallmatrix} \right) \in {\mathcal D}_{I}$
and $\left(\begin{smallmatrix} f_2 \\ f_c \\ 
    e_2 \\ - e_c \end{smallmatrix} \right) \in {\mathcal D}_{II}$, respectively.
Thus we see that the total power of the interconnected Dirac structure
is zero. This is not sufficient to show that ${\mathcal D}$
defined by (\ref{eq:1.3.25}) is a Dirac structure. However, it only remains to show that ${\mathcal D}$ is
maximal. For many coupled Dirac structures this holds.  If the systems
has the Hamiltonian $H_1$ and $H_2$ respectively, then the Hamiltonian
of the coupled system is $H_1 + H_2$. Of course we can extend this to
the coupling of more than two systems.

\section{Well-posedness}

In this section we investigate existence and uniqueness of solutions of the port-Hamiltonian system \eqref{eq:1.2.8}.
We define the boundary flow and effort as, see (\ref{eq:2HZ}),
$$  \begin{pmatrix}  f_\partial (t) \\ e_\partial(t) \end{pmatrix}=  \frac{1}{\sqrt{2}} \begin{bmatrix}
        P_1 & -P_1 \\ 
        I & I \end{bmatrix} \begin{pmatrix} {\mathcal
        H}(b) x(b,t) \\ {\mathcal
        H}(a) x(a,t) \end{pmatrix}.$$
Since $P_1$ is invertible, so is the matrix $\left[\begin{smallmatrix}
        P_1 & -P_1 \\ 
        I & I \end{smallmatrix}\right]$. Hence we can equivalently express the boundary conditions in terms of boundary effort and boundary flow. Thus we consider the system 
\begin{align}\nonumber
  \frac{\partial x}{\partial t}(\zeta,t) &= P_1 \frac{\partial }{\partial \zeta} \left({\mathcal H}(\zeta) x(\zeta,t) \right) + P_0 \left( {\mathcal H}(\zeta) x(\zeta,t) \right), \quad \zeta\in (a,b), t>0,\\
x(\zeta,0) &= x_0(\zeta), \qquad \zeta\in (a,b), \nonumber\\
 u(t) &=W_{B,1} \begin{pmatrix}  f_\partial (t) \\ e_\partial(t) \end{pmatrix},\qquad t>0,\label{eqn:phs}\\
 0 &=W_{B,2} \begin{pmatrix}  f_\partial (t) \\ e_\partial(t) \end{pmatrix},\qquad t>0,\nonumber\\
        y(t) &=W_{C} \begin{pmatrix}  f_\partial (t) \\ e_\partial(t) \end{pmatrix},\qquad t>0.\nonumber
\end{align}
The assumption on the matrices $P_0$, $P_1$ and ${\mathcal H}(\zeta)$ are as in Section~\ref{SEC:2.4}.
Furthermore,  $W_{B,1}$ is a $m\times 2n$-matrix, $W_{B,2}$ is a $(n-m)\times 2n$-matrix, 
 and $ W_C$ a $m\times 2n$-matrix. 
 Here
$u(t)\in \mathbb R^m$ denotes the input  and  $y(t)\in \mathbb R^m$ the output  at time $t$. 
Let  
$$W_B:=\left[\begin{smallmatrix}
        W_{B,1} \\ W_{B,2}   \end{smallmatrix}\right].$$
In the following we assume that the matrix $\left[\begin{smallmatrix}
        W_B \\ W_C \end{smallmatrix}\right]$ has full row rank. 
The Hamiltonian (energy)  of the port-Hamiltonian system is given by
\begin{equation*}
H(x(\cdot,t)) =
\frac{1}{2}\int_a^b x(\zeta,t)^\top{\mathcal
H}(\zeta)x(\zeta,t)\,d\zeta,
\end{equation*}
and an easy calculation shows that, see also \eqref{eq:1.2.10},
\begin{equation*}
\frac{d}{dt}H(x(\cdot,t)) =  \frac{1}{2} \left[  \left[{\mathcal H}
  x\right]^\top (\zeta,t) P_1  \left[{\mathcal H} x\right](\zeta,t) \right]_a^b
= e_\partial(t)^\top f_\partial(t).
\end{equation*}
A general assumption of a port-Hamiltonian systems is that they are \textit{impedance passive}, that is, 
\begin{equation*}
\frac{d}{dt}H(x(\cdot,t)) \le  u(t)^\top y(t).
\end{equation*}
If $m=n$, then this is satisfied if and only if 
$$\begin{bmatrix}  W_B\Sigma W_B^\top &  W_B\Sigma W_C^\top \\  W_C\Sigma W_B^\top &  W_C\Sigma W_C^\top \end{bmatrix}^{-1}\le \begin{bmatrix}  0 & I \\ I & 0 \end{bmatrix},$$
 where $\Sigma=\left[\begin{smallmatrix}  0 & I \\ I & 0 \end{smallmatrix}\right]$.

As \textit{state space} we choose  $X:=L^2((a,b);\mathbb{R}^n)$ equipped with the (energy) inner product
$
\langle x,y\rangle_X := \frac{1}{2}\int_a^b x(\zeta)^\top{\mathcal
H}(\zeta)y(\zeta)\,d\zeta$. The induced norm
 is equivalent to the standard $L^2$-norm.

First, we investigate solvability of the homogeneous port-Hamiltonian equation, that is, we assume $u=0$ and do not consider the output function $y$.

We define the system operator $A:{\mathcal D}(A)\subset X\to X$  by 
\begin{equation}\label{operatorA}
 Ax:= ( P_1 \frac{d}{d\zeta} + P_0) ({\mathcal H}x), \qquad x\in {\mathcal D}(A), 
\end{equation}
\begin{equation}
\label{domainA}
{\mathcal D}(A) := \left\{ x\in X\mid  {\mathcal H}x\in H^{1}((a,b);\mathbb R^n) \text{ and } W_B\left[\begin{smallmatrix} f_\partial\\e_\partial  \end{smallmatrix}\right]=0 \right\}.
\end{equation}
Identifying $x(\zeta,t)=x(t)(\zeta)$,  we  can rewrite our port-Hamiltonian system \eqref{eqn:phs} with $u=0$ and without output $y$ as
\begin{align}\label{eqn:cauchy}
\dot{x}(t)=Ax(t), \quad t\ge 0, \qquad x(0)=x_0.
\end{align}

We call a function $x:[0,\infty)\rightarrow X$ a \emph{classical solution of \eqref{eqn:cauchy}}, if $x$ is continuously differentiable, $x(t)\in D(A)$ for all $t\ge 0$ and \eqref{eqn:cauchy} is satified. Further, we call a function $x:[0,\infty)\rightarrow X$ a \emph{weak solution of \eqref{eqn:cauchy}}, if $x$ is continuously and
for every continuously differentiable functions $g:[0,\infty)\rightarrow X$ with $g(t)\in D(A^*)$ for all $t\ge 0$ and $Ag\in L^1([0,\tau];X)$ for every $\tau>0$ we have
\begin{align*}
\int_0^\tau \langle \dot{g}(t)+A^*g(t),x(t)\rangle dt =\langle g(\tau),x(\tau)\rangle-\langle g(0),x_0\rangle,\quad \tau >0.
\end{align*}
Clearly, every classical solution is a weak solution.
Our first main result concerning existence and uniqueness of solution is as follows.

\begin{theorem}[\cite{ZGMV10}]\label{theo1}
The port-Hamitionan system 
\begin{align*}\nonumber
  \frac{\partial x}{\partial t}(\zeta,t) &= P_1 \frac{\partial }{\partial \zeta} \left({\mathcal H}(\zeta) x(\zeta,t) \right) + P_0 \left( {\mathcal H}(\zeta) x(\zeta,t) \right), \quad \zeta\in (a,b), t>0,\\
x(\zeta,0) &= x_0(\zeta), \qquad \zeta\in (a,b), \nonumber\\
 0 &=W_{B} \begin{pmatrix}  f_\partial (t) \\ e_\partial(t) \end{pmatrix},\qquad t>0,
 \end{align*}
has for every initial condition $x_0\in X$ a unique weak solution with non increasing Hamiltonian if and only if
$$  W_B\Sigma W_B^\top\ge 0.$$
\end{theorem}

Next, we study the solvability of the  port-Hamiltonian system \eqref{eqn:phs}.
We say that the port-Hamiltonian system \eqref{eqn:phs} is \emph{well-posed}, if there are  $t_0, m_{t_0}>0$ such that  every classical solution of \eqref{eqn:phs} satisfies
$$\|x(t_0)\|^2_X +\int_0^{t_0} \|y(t)\|^2dt \le m_{t_0}\left[  \|x(0)\|^2_X +\int_0^{t_0} \|u(t)\|^2dt \right].$$
If the port-Hamiltonian system \eqref{eqn:phs} is \emph{well-posed}, then for every initial condition $x_0\in X$ and input function $u\in L^2_{loc}([0,\infty),\mathbb R^m)$ the system has an unique (weak) solution $x\in C([0,\infty),X)$ and $y\in L^2_{loc}([0,\infty),\mathbb R^m)$.

We remark, that for almost every $\zeta\in[a,b]$ the matrix $P_1\mathcal{H}(\zeta)$ can be diagonalized as $P_1\mathcal{H}(\zeta)=S^{-1}(\zeta)\Delta(\zeta)S(\zeta)$, where 
$\Delta(\zeta)$ is a diagonal matrix and  $S(\zeta)$ is an invertible matrix  for a.e.~$\zeta\in [a,b]$.

\begin{theorem}[\cite{ZGMV10}]\label{theo2}
If   $S^{-1}$, $S$,  $\Delta: [a,b] \rightarrow \mathbb R^{n\times 
n}$ are continuously differentiable and $W_B\Sigma W_B^\top\ge 0$, then 
the port-Hamiltonian system \eqref{eqn:phs} is well-posed.
\end{theorem}

\begin{example}[Lossless transmission line] 
We continue our Example \ref{ex1} of the lossless transmission line. Boundary effort and boundary flow are given by
\begin{align*}
\begin{pmatrix}   f_\partial \\ e_\partial  \end{pmatrix} =\frac{1}{\sqrt{2}}\begin{pmatrix}- \frac{x_2}{L}(b)+\frac{x_2}{L}(a) \\ - \frac{x_1}{C}(b)+\frac{x_1}{C}(a)\\  \frac{x_1}{C}(b)+\frac{x_1}{C}(a)\\  \frac{x_2}{L}(b)+\frac{x_2}{L}(a)\end{pmatrix}
\end{align*}
and we calculate
\begin{align*}
W_B =\frac{1}{\sqrt{2}}\begin{bmatrix}  0 &1&1&0\\ -R &-1&1&R\end{bmatrix}.
\end{align*}
An easy calculation shows
$$ W_B \Sigma W_B^\top=  \begin{bmatrix}  0 &0\\ 0 &2R\end{bmatrix} \ge 0.$$
Thus, by Theorem \ref{theo1} the lossless transmission line has for every initial condition $x_0\in X$ and $u=0$ a unique weak solution $x$ with non increasing Hamiltonian.
Further, we obtain
\begin{align*}
P_1 {\mathcal H} = S^{-1}\Delta S = \begin{bmatrix}  \gamma & -\gamma\\ \frac{1}{C} &\frac{1}{C}\end{bmatrix}\begin{bmatrix}  -\gamma &0\\ 0 &\gamma \end{bmatrix}\begin{bmatrix}  \frac{1}{2\gamma}&\frac{C}{2}\\ -\frac{1}{2\gamma} &\frac{C}{2}\end{bmatrix},
\end{align*}
where $\gamma$ is positive with $\gamma^2=\frac{1}{CL}$. Assuming that the functions $C$ and $L$ are positive and continuously differentiable, Theorem \ref{theo2} implies that the lossless transmission line  is well-posed. Further, inequality \eqref{eq:1.3} shows that the lossless transmission line is impedance passive.
\end{example}

\section{Stability}

This section is devoted to stability of linear (homogeneous) port-Hamiltonian systems
\begin{align}\nonumber
  \frac{\partial x}{\partial t}(\zeta,t) &= P_1 \frac{\partial }{\partial \zeta} \left({\mathcal H}(\zeta) x(\zeta,t) \right) + P_0 \left( {\mathcal H}(\zeta) x(\zeta,t) \right), \quad \zeta\in (a,b), t>0,\\
x(\zeta,0) &= x_0(\zeta), \qquad \zeta\in (a,b), \label{eqn:phshom}\\
 0 &=W_{B} \begin{pmatrix}  f_\partial (t) \\ e_\partial(t) \end{pmatrix},\qquad t>0.\nonumber
 \end{align}
The assumption on $P_0, P_1, {\mathcal H}, W_B$ are as in the previous section.
We study the question whether the solution of \eqref{eqn:phshom} tends to zero as time tends to infinity. We further assume that  $W_B\Sigma W_B^\top\ge 0$, which implies that for every $x_0\in X$ the port-Hamiltonian system \eqref{eqn:phshom} has a unique weak solution.

\begin{definition}
The port-Hamiltonian system \eqref{eqn:phshom}
 is {\em exponentially stable} if there exist 
constants $M>0$ and $\omega >0$ such that for every $x_0\in X$ the corresponding weak solution $x$ satisfies
  \begin{equation*}%
    \| x(t) \| \leq Me^{-\omega t}\|x_0\| \qquad \mbox{for } t \geq 0.
  \end{equation*}
\end{definition}
\begin{theorem}[{\!\!\cite{VZGM09}, \cite[Theorem 9.1.3]{JaZw12}}] \label{T:9.2.3}
   If for some positive constant $k$
  one of the following conditions is satisfied 
  \begin{itemize}
  \item for all $x \in X$ with  ${\mathcal H}x\in H^{1}((a,b);\mathbb R^n)$  and  $W_B\left[\begin{smallmatrix} f_\partial\\e_\partial  \end{smallmatrix}\right]=0$ we have 
  \[
     \left({\mathcal H} x\right)^\top (b) P_1 \left({\mathcal H} x\right)(b) - \left({\mathcal H} x\right)^\top (a) P_1 \left({\mathcal H} x\right)(a)  \leq -k
   \|{\mathcal H}(b)x(b)\|^2, \mbox{ or }
  \]
  \item for all $x \in X$ with  ${\mathcal H}x\in H^{1}((a,b);\mathbb R^n)$  and  $W_B\left[\begin{smallmatrix} f_\partial\\e_\partial  \end{smallmatrix}\right]=0$ we have  
  \[
      \left({\mathcal H} x\right)^\top (b) P_1 \left({\mathcal H} x\right)(b) - \left({\mathcal H} x\right)^\top (a) P_1 \left({\mathcal H} x\right)(a) \leq -k
   \|{\mathcal H}(a)x(a)\|^2,
   \]
  \end{itemize}
  then the port-Hamiltonian system \eqref{eqn:phshom}
 is exponentially stable.
\end{theorem}

By equation (\ref{eq:1.2.10}) we see that the left hand side can be regarded as twice the power flow at the boundary, or as twice the change of internal energy. So the above conditions can be interpreted as that the power flow at the boundary should be less than a negative constant times the energy at one of the boundaries.
\begin{example}
 We consider the lossless transmission line on the spatial interval $[a,b]$ as discussed in Examples \ref{ex1} with $u(t)=0$ and $R>0$. Let  $x \in X$ with  ${\mathcal H}x\in H^{1}((a,b);\mathbb R^n)$  and  $W_B\left[\begin{smallmatrix} f_\partial\\e_\partial  \end{smallmatrix}\right]=0$.  Using  \eqref{eq:1.2.10} and \eqref{eq:1.3} we get
\[
   \frac{1}{2} \left[ \left({\mathcal H}
  x\right)^\top (\zeta) P_1  \left({\mathcal H} x\right)(\zeta) \right]_a^b
  = \frac{x_2}{L}\left(a\right)\frac{x_1}{C}\left(a\right) - \frac{x_2}{L}\left(b\right)\frac{x_1}{C}\left(b\right).
  \]
As $W_B\left[\begin{smallmatrix} f_\partial\\e_\partial  \end{smallmatrix}\right]=0$, which in particular implies
$\frac{x_1}{C}\left(a\right)=0$ and 
$\frac{x_1}{C}\left(b\right)=R\frac{x_2}{L}\left(b\right)$, we obtain 
\begin{align*}
  \left({\mathcal H} x\right)^\top (b) P_1 \left({\mathcal H} x\right)(b) &- \left({\mathcal H} x\right)^\top (a) P_1 \left({\mathcal H} x\right)(a)  \\
  =&  - R \left(\frac{x_2}{L}\left(b\right)\right)^2   =  - \frac{R}{1+R} \|{\mathcal H}(b)x(b)\|^2 .
 \end{align*}
 Thus, if $R>0$, then by Theorem \ref{T:9.2.3}   the  lossless transmission line is exponentially stable.
\end{example}

\section{From the transmission line to the heat equation}

As seen in Section \ref{sec:1.3}, given a Dirac structure we can make different choices for the effort and flow variables. We illustrate this with the Dirac structure associated to the transmission line of Example \ref{E:2.4.3}. 

Thus we consider ${\mathcal E} = {\mathcal F} = L^2((a,b);{\mathbb R}^2) \times {\mathbb R}^2$, and 
\begin{align}
  \nonumber
    {\mathcal D} =&\ \left\{  \left( \begin{array}{c} f \\ f_{\partial}\\ e\\ e_{\partial}  \end{array} \right) \in {\mathcal F} \times {\mathcal E} \mid f \in L^2((a,b);{\mathbb R}^2), e  \in H^1((a,b);{\mathbb R}^2) \right.\\
      \label{eq:6HZ}
    &\qquad  \left. \left( \begin{array}{c} f_1 \\ f_{2}  \end{array} \right)= \left( \begin{array}{c} -\frac{de_2}{d\zeta} \\  -\frac{de_1}{d\zeta} \end{array} \right), \left( \begin{array}{c}  f_{\partial,1} \\ f_{\partial,2}  \\  e_{\partial,1} \\ e_{\partial,2}\end{array} \right) =
   \frac{1}{\sqrt{2}} \left( \begin{array}{c} -e_2(b)+e_2(a) \\ -e_1(b)+e_1(a)\\ e_1(b)+e_1(a)\\ e_2(b)+e_2(a)
   \end{array} \right)
    \right\}
  \end{align}
 This is the Dirac structure (\ref{eq:2HZ}) with $P_1 = \left[\begin{smallmatrix} 0 &-1\\ -1 &0 \end{smallmatrix} \right]$, and $P_0 =0$. 
  
 Instead of the choices (\ref{eq:3HZ}), we choose
 \begin{equation}
 \label{eq:7HZ}
   f_1= \dot{x},\quad f_2 = \frac{e_2}{\alpha} , \quad e_1= h x,
 \end{equation}
 where $\alpha$ and $h$ are functions in $L^{\infty} (a,b)$ which are positive and $\alpha^{-1}$, $h^{-1}$ are in $L^{\infty} (a,b)$ as well.
 
With this choice we obtain the partial differential equation 
 \begin{equation*}
 \label{eq:8HZ}
   \dot{x} = f_1 = - \frac{de_2}{d \zeta} = - \frac{d}{d\zeta} \left[ \alpha f_2 \right] = - \frac{d}{d\zeta} \left[ -\alpha \frac{de_1}{d \zeta} \right] = \frac{d}{d\zeta} \left[ \alpha \frac{d}{d \zeta} \left[ hx\right] \right] ,
 \end{equation*}
 where we combined the choices from (\ref{eq:7HZ}) with the definition of $f_1,f_2$ from (\ref{eq:6HZ}). For $h=1$, the associated partial differential equation is the heat equation, 
 \begin{equation}
 \label{eq:9HZ}
   \frac{\partial x}{\partial t}(\zeta,t) = \frac{\partial}{\partial\zeta} \left[ \alpha(\zeta) \frac{\partial}{\partial \zeta} \left[ x(\zeta,t)\right] \right], \qquad \zeta \in (a,b), t\geq 0. 
 \end{equation}
We know that we can express our boundary conditions in the boundary flow and effort. Using (\ref{eq:6HZ}) and (\ref{eq:7HZ}), it yields
\[
  \left( \begin{array}{c}  f_{\partial,1} \\ f_{\partial,2}  \\  e_{\partial,1} \\ e_{\partial,2}\end{array} \right) =
   \frac{1}{\sqrt{2}} \left( \begin{array}{c} \alpha(b) \frac{d hx}{d \zeta} (b)- \alpha(a) \frac{d hx}{d \zeta}(a) \\ -h(b)x(b)+h(a)x(a)\\ h(b)x(b)+h(a)x(a)\\ -\alpha(b) \frac{d hx}{d \zeta}(b)- \alpha(a) \frac{d hx}{d \zeta}(a)
   \end{array} \right)
\]
Combining (\ref{eq:1.3.14}) with (\ref{eq:1HZ}), we obtain
\begin{align*}
  f_{\partial}^Te_{\partial} =&\ \int_a^b f(\zeta)^T e(\zeta) d\zeta\\
   =&\  \int_a^b f_1(\zeta)e_1(\zeta) +  f_2(\zeta)e_2(\zeta) d\zeta\\
    =&\ \int_a^b \dot{x}(\zeta,t) h (\zeta) x(\zeta,t ) d\zeta + \int_a^b f_2(\zeta) \alpha(\zeta) f_2(\zeta) d\zeta .
 \end{align*}
 If we define the ``energy'' for the PDE (\ref{eq:9HZ}) as $$H(t) = \int_a^b x(\zeta,t) h (\zeta) x(\zeta,t ) d\zeta,$$ then 
 \[
   \dot{H}(t) = 2 \int_a^b \dot{x}(\zeta,t) h (\zeta) x(\zeta,t ) d\zeta = 2 f_{\partial}^Te_{\partial} - 2\int_a^b f_2(\zeta) \alpha(\zeta) f_2(\zeta) d\zeta \leq 2 f_{\partial}^Te_{\partial},
 \]
 where we have used that $\alpha$ is positive. Thus,  these system have internal dissipation, i.e, when we have chosen the boundary conditions such that $f_{\partial}^Te_{\partial}=0$, we still have that $H$ will decrease. Therefore we have shown that Dirac structures may be associated to systems with dissipation.  The relation between $f_2$ and $e_2$ can be regarded as adding a {\em closure relation} or {\em resistive} part to the Dirac structure, as is shown below.
 \begin{figure}[ht]
  \centering
   \includegraphics[scale=1.0]{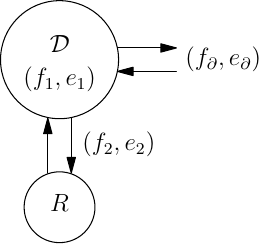}
  \caption{A Dirac structure wit a resistive part}
  \label{fig:1.4}
\end{figure}

 \section {Riesz basis  for port-Hamiltonian systems}
 
 A well-known technique for finding a non-zero solution of a homogeneous linear partial differential equation, is the method of separation of variables. Thus a solution is sought of the form $x(t,\zeta) = f(t) g(\zeta)$ which satisfies the boundary conditions, but not necessarily the initial condition. When a collection of these solutions are found, then by the linearity of the partial differential equation, an arbitrary sum still satisfies the equation and the boundary conditions. For the heat and wave equation this leads to the general form of the solution. 
 
The separation of variables method is usually applied to scalar homogeneous linear partial differential equations, but it can be easily extended to our class of port-Hamiltonian systems. To keep notation simple we use the abstract differential equation notation $\dot{x}(t) =Ax(t)$, see (\ref{eqn:cauchy}) with $A$ and its domain given by (\ref{operatorA}) and (\ref{domainA}) respectively. 

For the equation $\dot{x}(t) = A x(t)$ we assume that 
\begin{equation}
\label{eq:11HZ}
  x(t) = f(t) \phi(\zeta), \mbox{ with } f: [0,\infty) \mapsto {\mathbb R}, \mbox{ and } \phi: [a,b] \mapsto {\mathbb R}^n.
\end{equation}
This leads to the equation
\[
  \dot{f}(t) \phi(\zeta) = f(t) (A\phi)(\zeta).
\]
Assuming $f(t) \neq 0$ for all $t$, we get the equivalent equation
\[
  \frac{\dot{f}(t)}{f(t)}  \phi(\zeta) = (A\phi)(\zeta).
\]
Now the left-hand side depends on $t$ and $\zeta$, whereas the right-hand side only depends on $\zeta$. Thus the expression $\dot{f}/f$ may not depend on $t$, and so it is a constant. We call this constant $\lambda$, and we get the following two equations
\[
  \dot{f}(t) = \lambda f(t) \mbox{ and } \lambda \phi(\zeta) = (A\phi)(\zeta).
\]
The first one is easy to solve, and gives us that $f(t) = f_0 e^{\lambda t}$. The second one becomes
\begin{equation}
\label{eq:12HZ}
  P_1 \frac{d{\mathcal H} \phi}{d\zeta} + P_0 {\mathcal H} (\zeta) \phi(\zeta) = \lambda \phi(\zeta), \quad W_B \left[\begin{array}{cc} P_1 & -P_1 \\ I & I \end{array} \right]\left(\begin{matrix}  {\mathcal H} (b) \phi(b) \\   {\mathcal H} (a) \phi(a) \end{matrix} \right) =0.
\end{equation}
The last equation has to hold, since $x$ is assumed to be a classical solution, thus an element in the domain of $A$. Since $f$ is a scalar only depending on $t$, we see that the boundary conditions of $x$ transfer to $\phi$. Like for matrices, the equation $A\phi=\lambda \phi$ is called an {\em eigenvalue/eigenfunction} equation.

Summarising the above we see that a non-zero solution of the form (\ref{eq:11HZ}) exists if and only if there exists a non-zero solution of (\ref{eq:12HZ}) for some $\lambda$. In (\ref{eq:12HZ}) we have an ordinary differential equation and boundary conditions. It is well-known that this differential equation will always have a solution for any $\lambda$, and so the boundary conditions will form the restricting condition.  For instance, when $P_1=1, P_0=0$, and ${\mathcal H}=1$, we see that 
\[ 
  \phi(\zeta) = \phi_0 e^{\lambda \zeta}
\]
is the solution of the differential equation $\phi' =\lambda\phi$. When the boundary condition equals $\phi(a) =0$, then there exists no $\lambda $ and non-zero $\phi_0$ for which (\ref{eq:12HZ}) holds. 

However, when the boundary conditions is given as $\phi(a) = \phi(b)$, then we see that we have infinitely many solutions. Namely,
\[
   \phi(a) = \phi(b) \Leftrightarrow e^{\lambda a} = e^{\lambda b}  \Leftrightarrow \lambda = \frac{2n\pi i}{b-a},~ n \in {\mathbb Z}.
\]   
In the second situation we find as solutions of (\ref{eq:11HZ})
\[
  x_n(t,\zeta) = a_n e^{\lambda_n t} e^{\lambda_n \zeta}, \mbox{ with } \lambda_n = \frac{2n\pi i}{b-a},~ n \in {\mathbb Z}
\]
Since our PDE is linear, we have that $x(t,\zeta) = \sum_{n=-N}^N a_n e^{\lambda_n t} e^{\lambda_n \zeta}$ is also a solution. The initial condition for this solution is $x(0,\zeta) = \sum_{n=-N}^N a_ne^{\lambda_n \zeta}$. Now the last expression reminds us of the complex Fourier series, which says that every function $x_0 \in L^2(a,b)$ can be written as
\[
  x_0(\zeta) = \sum_{n=-\infty}^{\infty} a_ne^{\lambda_n \zeta}
\]
with $a_n$ the Fourier coefficients which form an $\ell^2({\mathbb Z})$-sequence. Based on this we assert that 
\[
  x(t,\zeta) = \sum_{n=-\infty}^{\infty} a_n e^{\lambda_n t} e^{\lambda_n \zeta}
\]
is the solution of our partial differential equation for the given $x_0$. This is indeed the case, but it is a weak solution. 

Thus depending on the boundary conditions we could have no solution of (\ref{eq:11HZ}) or infinitely many which even form the basis for all solutions of this equation. The following recent theorem shows that we have a basis of all solutions if and only if the partial differential equation possesses a solution forward and backwards in time. 
Next we characterize all boundary conditions such that the all solutions of the port-Hamiltonian system
\begin{align}\nonumber
  \frac{\partial x}{\partial t}(\zeta,t) &= P_1 \frac{\partial }{\partial \zeta} \left({\mathcal H}(\zeta) x(\zeta,t) \right) + P_0 \left( {\mathcal H}(\zeta) x(\zeta,t) \right), \quad \zeta\in (a,b), t>0,\\
x(\zeta,0) &= x_0(\zeta), \qquad \zeta\in (a,b),\label{eqn:phsriesz} \\
 0 &=\widetilde W_{B} \begin{pmatrix} {\mathcal
        H}(b) x(b,t) \\ {\mathcal
        H}(a) x(a,t) \end{pmatrix},\qquad t>0,\nonumber
\end{align}
can be represented as a series of basis functions. Equivalently, we characterize all port-Hamiltonian systems \eqref{eqn:phsriesz} such that the eigenspace of the corresponding system operator $A$ form a  Riesz basis of subspaces\footnote{We have to consider subspaces, i.e., eigenspaces, because the eigenvalue could have multiplicity two or higher. Furthermore we allow for generalised eigenfunctions as well.} of the state space $X$. Again 
 $P_1\in \mathbb R^{n\times n}$ is invertible and self-adjoint, $P_0\in \mathbb R^{n\times n}$  is  skew-adjoint,  for every $\zeta \in [a,b]$ the matrix
${\mathcal H}(\zeta)$ is self-adjoint  and there exist constants $c,C>0$, such that $cI \leq
{\mathcal H}(\zeta) \leq C I$ for almost every $\zeta\in[a,b]$, and   $\widetilde W_{B}$ is a $n\times 2n$-matrix with full row rank. We split $\widetilde W_B=\left[\begin{smallmatrix} W_b & W_a \end{smallmatrix}\right]$. By $Z^-(b)$ we denote the span of the eigenvectors of $P_1{\mathcal H}(b)$ corresponding to its negative eigenvalues and 
$Z^+(a)$ is the span of the eigenvectors of $P_1{\mathcal H}(a)$ corresponding to its positive eigenvalues.

\begin{theorem}[\cite{JacKai21}]
The following are equivalent:
\begin{enumerate}
 \item The eigenspace of the  system operator $A$ of \eqref{eqn:phsriesz} form a  Riesz basis of subspaces of the state space $X$.
 \item \begin{enumerate}
\item For every $x_0\in X$ the port-Hamiltonian system \eqref{eqn:phsriesz} has a unique weak solution.
\item  For every $x_0\in X$ the port-Hamiltonian system \eqref{eqn:phsriesz} with $P_0$ replaced by $-P_0$ and $P_1$ replaced by $-P_1$ has a unique weak solution.
\end{enumerate}
\item $ W_b{\mathcal H}(b)Z^-(b)\oplus W_a{\mathcal H}(a)Z^+(a)= W_b{\mathcal H}(b)Z^+(b)\oplus W_a{\mathcal H}(a)Z^-(a)=\mathbb R^n$.
\end{enumerate}
\end{theorem}

\section{Exercise}

\begin{example}
\label{E7.1.1}
We consider the {\em vibrating string} as depicted in Figure \ref{fig:wave1-bis}.
\begin{figure}[htb]
    \centering
   \includegraphics[scale=0.7]{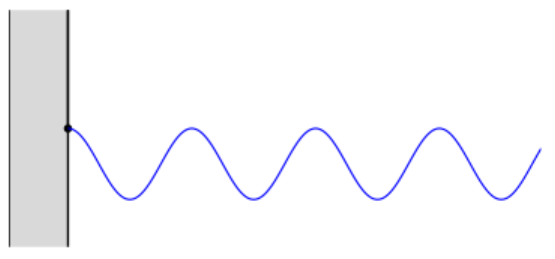}
    \caption{The vibrating string}
    \label{fig:wave1-bis}
\end{figure}
The string is fixed at the left hand-side. We allow that a force $u$ may be applied at the right hand-side and as output $y$ we choose the velocity  at the right hand side. The model of the (undamped) vibrating string is given by
  \begin{equation}
  \label{eq:7.1.1}
    \frac{\partial^2 w}{\partial t^2} (\zeta,t) = \frac{1}{\rho(\zeta)} \frac{\partial }{\partial \zeta} \left( T(\zeta) \frac{\partial w}{\partial \zeta}(\zeta,t) \right),
  \end{equation}
 where $\zeta\in [a,b]$ is the spatial variable, $w(\zeta,t)$
  is the vertical position of the string at place $\zeta$ and time $t$, $T$ is the Young's modulus
  of the string, and $\rho$ is the mass density, which may vary along the string. We assume that both functions $\rho$ and $T$ are continuously differentiable.
This system has the energy/Hamiltonian
  \begin{equation}
  \label{eq:7.1.2}
    E(t) = \frac{1}{2} \int_a^b \rho(\zeta) \left(\frac{\partial w}{\partial t}(\zeta,t)\right)^2 + T(\zeta) \left(\frac{\partial w}{\partial \zeta}(\zeta,t)\right)^2 d \zeta.
  \end{equation} 
\vspace{1ex}

\begin{enumerate}
\item Formulate the wave equation as port-Hamiltonian system. What are the matrices $P_0$, $P_1$, ${\mathcal H}$, $\widetilde W_B$ and $\widetilde W_C$?
\item Show that for $u=0$ and every initial condition the wave equation has a unique weak solution.
\item Show that the wave equation (with inputs and outputs) is well-posed.
\item We now attach a damper at the right hand side, that is, the force at that end equals a (negative) constant times the velocity at that end, i.e.
$$ T(b)\frac{\partial w}{\partial \zeta}(b,t)=-k\frac{\partial w}{\partial t}(b,t), \qquad k>0.$$
Show that the resulting wave equation is exponentially stable.
\end{enumerate}

\end{example}

 \section*{Conclusions and further results}
 
  In this article we presented an introduction to infinite-dimen\-sional systems theory. However, we only consider port-Hamiltonian systems, where the system operator 
  is given by $P_1 \frac{d }{d \zeta} \left[ {\mathcal H} x \right] + P_0 \left[{\mathcal H} x\right]$.
  In order to model examples like the Euler-Bernoulli beam, Schr\"odinger equation or Airy's equation, more general port-Hamiltonian systems of the form 
  \begin{align*}
   \frac{\partial x}{\partial t}(\zeta,t) = \sum_{j=0}^NP_j \frac{\partial^j }{\partial \zeta^j} \left[ {\mathcal H}(\zeta) x(\zeta,t) \right] 
   \end{align*}
  need to be investigated. Several results extend to this more general class of port-Hamiltonian systems. The contraction semigroup generation results have been shown by Le Gorrec, Zwart and Maschke \cite{GoZM05}. Further, stability has been investigated in  Augner and Jacob \cite{AuJa14}. 
Infinite-dimensional port-Hamiltonian system on an multi-dimensional spatial domain have been investigated in Kurula and Zwart \cite{KuZw15} and Skrepek \cite{Sk21}.

\def\cprime{$'$}


\begin{thebibliography}{10}

\bibitem{AuJa14}
B.~Augner and B.~Jacob.
\newblock Stability and stabilization of infinite-dimensional linear
  port-{H}amiltonian systems.
\newblock {\em Evol. Equ. Control Theory}, 3(2):207--229, 2014.

\bibitem{BCEJLLM09}
A.~Baaiu, F.~Couenne, D.~Eberard, C.~Jallut, L.~Lef\`{e}vre, Y.~{Le Gorrec},
  and B.~Maschke.
\newblock Port-based modelling of mass transport phenomena.
\newblock {\em Math. Comput. Model. Dyn. Syst.}, 15(3):233--254, 2009.

\bibitem{CeSB07}
J.~Cervera, A.~J. van~der Schaft, and A.~Ba{\~n}os.
\newblock Interconnection of port-{H}amiltonian systems and composition of
  {D}irac structures.
\newblock {\em Automatica J. IFAC}, 43(2):212--225, 2007.

\bibitem{EbMS07}
D.~Eberard, B.~M. Maschke, and A.~J. van~der Schaft.
\newblock An extension of {H}amiltonian systems to the thermodynamic phase
  space: towards a geometry of nonreversible processes.
\newblock {\em Rep. Math. Phys.}, 60(2):175--198, 2007.

\bibitem{HDLM10}
B.~Hamroun, A.~Dimofte, L.~Lef{\`e}vre, and E.~Mendes.
\newblock Control by interconnection and energy-shaping methods of port
  {H}amiltonian models. {A}pplication to the shallow water equations.
\newblock {\em Eur. J. Control}, 16(5):545--563, 2010.

\bibitem{JacKai21}
B.~Jacob, J.~T. Kaiser, and H.~Zwart.
\newblock Riesz bases of port-{H}amiltonian systems.
\newblock {\em SIAM J. Control Optim.}, 59(6):4646--4665, 2021.

\bibitem{JaZw12}
B.~Jacob and H.~Zwart.
\newblock {\em Linear port-{H}amiltonian systems on infinite-dimensional
  spaces}, volume 223 of {\em Operator Theory: Advances and Applications}.
\newblock Birkh\"{a}user/Springer Basel AG, Basel, 2012.
\newblock Linear Operators and Linear Systems.

\bibitem{JeSc09}
D.~Jeltsema and A.~J. van~der Schaft.
\newblock Lagrangian and {H}amiltonian formulation of transmission line systems
  with boundary energy flow.
\newblock {\em Rep. Math. Phys.}, 63(1):55--74, 2009.

\bibitem{KuZw15}
M.~Kurula and H.~Zwart.
\newblock Linear wave systems on {$n$}-{D} spatial domains.
\newblock {\em Internat. J. Control}, 88(5):1063--1077, 2015.

\bibitem{KZSB10}
M.~Kurula, H.~Zwart, A.~van~der Schaft, and J.~Behrndt.
\newblock Dirac structures and their composition on {H}ilbert spaces.
\newblock {\em J. Math. Anal. Appl.}, 372(2):402--422, 2010.

\bibitem{GoZM05}
Y.~{Le Gorrec}, H.~Zwart, and B.~Maschke.
\newblock Dirac structures and boundary control systems associated with
  skew-symmetric differential operators.
\newblock {\em SIAM J. Control Optim.}, 44(5):1864--1892 (electronic), 2005.

\bibitem{MM05}
A.~Macchelli and C.~Melchiorri.
\newblock Control by interconnection of mixed port {H}amiltonian systems.
\newblock {\em IEEE Trans. Automat. Control}, 50(11):1839--1844, 2005.

\bibitem{OvSME02}
R.~Ortega, A.~van~der Schaft, B.~Maschke, and G.~Escobar.
\newblock Interconnection and damping assignment passivity-based control of
  port-controlled {H}amiltonian systems.
\newblock {\em Automatica J. IFAC}, 38(4):585--596, 2002.

\bibitem{Sk21}
N.~Skrepek.
\newblock Well-posedness of linear first order port-{H}amiltonian systems on
  multidimensional spatial domains.
\newblock {\em Evol. Equ. Control Theory}, 10(4):965--1006, 2021.

\bibitem{vS06}
A.~van~der Schaft.
\newblock Port-{H}amiltonian systems: an introductory survey.
\newblock In {\em International {C}ongress of {M}athematicians. {V}ol. {III}},
  pages 1339--1365. Eur. Math. Soc., Z\"urich, 2006.

\bibitem{ScMa02}
A.~van~der Schaft and B.~Maschke.
\newblock Hamiltonian formulation of distributed-parameter systems with
  boundary energy flow.
\newblock {\em J. Geom. Phys.}, 42(1-2):166--194, 2002.

\bibitem{VZGM09}
J.~Villegas, H.~Zwart, Y.~{Le Gorrec}, and B.~Maschke.
\newblock Exponential stability of a class of boundary control systems.
\newblock {\em IEEE Trans. Automat. Control}, 54(1):142--147, 2009.

\bibitem{ZGMV10}
H.~Zwart, Y.~{Le Gorrec}, B.~Maschke, and J.~Villegas.
\newblock Well-posedness and regularity of hyperbolic boundary control systems
  on a one-dimensional spatial domain.
\newblock {\em ESAIM: Control, Optimisation and Calculus of Variations},
  16(4):1077--1093, Oct 2010.

\end{thebibliography}

\end{document}